\documentclass{amsart}
\usepackage{amsfonts}
\usepackage{amssymb}
\usepackage{amsmath}
\usepackage[pagewise]{lineno}\linenumbers
\newtheorem{theorem}{Theorem}[section]

\newtheorem{corollary}[theorem]{Corollary}

\newtheorem{examples}[theorem]{Examples}
\newtheorem{definition}[theorem]{Definition}

\newtheorem{example}[theorem]{Example}

\newtheorem{lemma}[theorem]{Lemma}

\newtheorem{proposition}[theorem]{Proposition}
\newtheorem{remarks}[theorem]{Remarks}

\newcommand{\nn}{\mathbb{N}}

\newcommand{\vf}{\varphi}

\font\msbm=msbm10 at 12pt

\newcommand{\ZZ}{\mbox{\msbm Z}}

\newcommand{\F}{\mbox{\msbm F}}

\begin{document}

\title{Scalar products  and  Left LCD codes }
\author{Nabil Bennenni}
\address{USTHB, Facult\'{e} de Math\'{e}matiques, Laboratoire ATN,\\
	BP 32 El Alia Bab Ezzouar Algiers Algeria\\
	email: nabil.bennenni@gmail.com}

\author{ Andr\'e Leroy}
\address{Univ. Artois, UR 2462, Laboratoire de Mathématiques de Lens (LML), F-62300 Lens, France\\
	email: andre.leroy@univ-artois.fr}



\date{}

\begin{abstract}
 In this article, we introduce new scalar products over  finite rings via additive isomorphisms.  This allows us to define new notions of right (respectively left) orthogonal codes, that are not necessarily linear.  This leads to definitions of right (resp. left) dual codes and left LCD codes similar to the classical LCD codes.  
 Furthermore, we provide necessary and sufficient conditions for the existence of these codes.
\end{abstract}

\maketitle



\section{Introduction}
Linear complementary dual codes (LCD) were initially introduced by Massey \cite{LCD2} in 1992, with the intention of using them for the so-called two-user binary adder channel. Massey proved that the use of an LCD code solves some of the decodability difficulties. LCD codes  play an important role in practical applications in particular, against side-channel and fault injection attacks; for more details cf. \cite{LCD4}.\\
In this paper, we focus on the class of left LCD  codes (left  Complementary Dual codes) which is important because of their connection with quantum error correction and quantum computing \cite{LCD5}.\\
Skew triangular matrix rings presented in \cite {LCD1} gave us the opportunity to create new dot products, which is extremely beneficial for creating left and right dual codes different from the classic dual codes. In addition, we stipulated that the codes must be left-LCD codes (left-linear Complementary Dual codes).\\
This paper is organized as follows. In Section 2, we define a new product on a matrix ring $M_n(R)$ via the action of an automorphism $\theta \in Aut(R)$.  This was inspired by \cite{LCD1} where a similar construction was defined on upper triangular matrix rings.
We also introduce new dot products and their left and right orthogonality relations.  We characterize self-dual codes with respect to this new dot product.   We study the connections between some subsets of $R^n$ arising from these definitions.
In Section 3, we study a necessary and sufficient condition for a left linear code to be a left LCD code (left linear Complementary Dual codes).  We give examples of the best known left LCD codes obtained in this way using $\mathbb{F}_4$, $\mathbb{F}_8$, $\mathbb{F}_9$  or $\mathbb{F}_{16}$ for $R$ and the Frobenius automorphism for $\theta$.

\section{New products Via additive isomorphisms}

We will define a new product on a matrix ring $M_n(R)$ via the action of an automorphism $\theta \in Aut(R)$.   Let us start with a very general statement.
\begin{lemma}
\label{Lemme general construction of new product}
	Let $R,+,.$ be a ring and $S,+$  be an additive, abelian group such that $\varphi: R,+\longrightarrow S,+$ is an isomorphism of additive groups.  Then the following multiplication $*$ gives $S$ a ring structure.
	$$
	For \; s_1,s_2 \in S, \; s_1*s_2=\vf(\vf^{-1}(s_1)\vf^{-1}(s_2))
	$$
	The map $\vf:R \rightarrow S$ is therefore a ring isomorphism.
\end{lemma}
\begin{proof}
Let $R,+,.$ be a ring and $S,+$ be an additive (=abelian) group such that $\varphi: R,+\longrightarrow S,+$ is an isomorphism of additive groups.  By definition, we  have $s_1*s_2=\vf(\vf^{-1}(s_1)\vf^{-1}(s_2)).$  We compute
\begin{align*}
s_1*(s_2*s_3)&=\vf(\vf^{-1}(s_1)\vf^{-1}(s_2*s_3))\\
&=\vf(\vf^{-1}(s_1)\vf^{-1}(\vf(\vf^{-1}(s_2)\vf^{-1}(s_3))))\\
&=\vf(\vf^{-1}(s_1)(\vf^{-1}(s_2)\vf^{-1}(s_3)))\\
&=\vf((\vf^{-1}(s_1)\vf^{-1}(s_2))\vf^{-1}(s_3))\\
&=(s_1*s_2)*s_3.
\end{align*}
\begin{align*}
s_1*(s_2+s_3)&=\vf(\vf^{-1}(s_1)\vf^{-1}(s_2+s_3))\\
&=\vf(\vf^{-1}(s_1)\vf^{-1}(s_2)+\vf^{-1}(s_1)\vf^{-1}(s_3))\\
&=\vf(\vf^{-1}(s_1)\vf^{-1}(s_2))+\vf(\vf^{-1}(s_1)\vf^{-1}(s_3))\\
&=s_1*s_2+s_1*s_3.
\end{align*}
\begin{align*}
s_1*\vf(1_R)&=\vf(\vf^{-1}(s_1)\vf^{-1}(\vf(1)))\\
&=\vf(\vf^{-1}(s_1)1)=s_1.
\end{align*}
\begin{align*}
\vf(r_1r_2)&=\vf(\vf^{-1}(s_1)\vf^{-1}(s_2))\\
&=s_1*s_2=\vf(r_1)\vf(r_2).
\end{align*}
\end{proof}

Before coming to the main focus of our paper, let us apply this lemma and give an example:

\begin{example}
	\begin{enumerate}
		\item Consider $\ZZ/n\ZZ$ and let 
		$m\in \{1,\dots ,n\}$ be such that $m$ and $n$ are coprime.  The map $\vf: \ZZ/n\ZZ,+\longrightarrow \ZZ/n\ZZ+$ defined by $\vf(1+n\ZZ)=m+n\ZZ$ is an isomorphism of additive structure.  This gives a new ring structure on $\ZZ/n\ZZ$ where the unity will be $m+n\ZZ$.
		\item  Consider $\ZZ/8\ZZ$ and the map $\vf: \ZZ/8\ZZ,+,*\longrightarrow \ZZ/8\ZZ,+,*$ defined by $\vf(\ZZ/8\ZZ)=3\times (\ZZ/8\ZZ)$.
		 $\overline{2}*\overline{3}=\vf(\vf^{-1}(\overline{2})\vf^{-1}(\overline{3}))=\vf(\overline{6}.\overline{1})=2$.\\
		$I(\ZZ/8\ZZ)=2.(\ZZ/8\ZZ)$, $\vf(I(\ZZ/8\ZZ))=\vf(2.(\ZZ/8\ZZ))=\vf(\overline{2})*\ZZ/8\ZZ=(\overline{6}*(\ZZ/8\ZZ).$
	\end{enumerate}

\end{example}

We can apply this lemma to a matrix ring $M_n(R)$ and an automorphism $\theta$ of $R,+$
(not necessarily a ring automorphism).   We get an additive isomorphism, denoted $\vf:M_n(R)\longrightarrow M_n(R)$, by setting, for $A=(A_{ij})\in M_n(R)$, $\vf(A)_{ij}=\theta^{i-1}(A_{ij})$.   From $\varphi$ we can create a new product on $M_n(R)$.  We describe this product in Theorem \ref{New Product on Matrices}.

\begin{theorem}
\label{New Product on Matrices}
	Let $\theta\in Aut(R)$, and $\vf$ the map defined above.   Then the product $*$ defined on $M_n(R),+$ satisfies, for $A=(a_{ij})$ and $B=(b_{ij})$ we have
	$$
	(A*B)_{ij}=\sum_{k=1}^n a_{ik}\theta^{i-k}(b_{kj})
	$$
\end{theorem}
\begin{proof}
For $1\le i, j\le n$, we compute:
\begin{align*}
	(A*B)_{ij}=\varphi(\varphi^{-1}(A)\varphi^{-1}(B))_{ij}&=\theta^{i-1}((\varphi^{-1}(A)\varphi^{-1}(B))_{ij})\\
	&=\theta^{i-1}(\sum_k\varphi^{-1}(A)_{ik}\varphi^{-1}(B)_{kj})\\
	&=\sum_k\theta^{i-1}(\theta^{1-i}(A_{ik})\theta^{1-k}(B_{kj}))\\
	&=\sum_kA_{ik}\theta^{i-k}(B_{kj}) .
	\end{align*}
	  This concludes the proof.	
\end{proof}
\begin{example}We define the additive map $\vf$ by: 
	\begin{align*}
\vf :M_2(R)&\longrightarrow M_2(R) \\
\left(
\begin{array}{cc}
a & b \\
c & d \\
\end{array}
\right)&\longmapsto \left(
\begin{array}{cc}
a & c \\
b & d \\
\end{array}
\right);	
\end{align*}
\begin{align*}
\left(
\begin{array}{cc}
a & b \\
c & d \\
\end{array}
\right)*\left(
\begin{array}{cc}
a' & b' \\
c' & d' \\
\end{array}
\right)=&\vf(\vf^{-1}(\left(
\begin{array}{cc}
a & b \\
c & d \\
\end{array}
\right))\vf^{-1}(\left(
\begin{array}{cc}
a' & b' \\
c' & d' \\
\end{array}
\right)))\\
=&\vf(\left(
\begin{array}{cc}
a & c \\
b & d \\
\end{array}
\right)\left(
\begin{array}{cc}
a' & c' \\
b' & d' \\
\end{array}
\right))\\
=& \vf(\left(
\begin{array}{cc}
aa'+cb' & ac'+cd' \\
ba'+db' & bc'+dd'\\
\end{array}
\right))\\
&=
\left(
\begin{array}{cc}
aa'+cb' & ba'+db'\\
ac'+cd'  & bc'+dd'\\
\end{array}
\right).	
\end{align*}
\end{example}
\begin{example}
 We define the additive map $\vf$ by: 
\begin{align*}
\vf :M_2(R)&\longrightarrow M_2(R) \\
\left(
\begin{array}{cc}
a & b \\
c & d \\
\end{array}
\right)&\longmapsto \left(
\begin{array}{cc}
b & c \\
d & a \\
\end{array}
\right);	
\end{align*}
\begin{align*}
\left(
\begin{array}{cc}
a & b \\
c & d \\
\end{array}
\right)*\left(
\begin{array}{cc}
a' & b' \\
c' & d' \\
\end{array}
\right)=&\vf(\vf^{-1}(\left(
\begin{array}{cc}
a & b \\
c & d \\
\end{array}
\right))\vf^{-1}(\left(
\begin{array}{cc}
a' & b' \\
c' & d' \\
\end{array}
\right)))\\
=&\vf(\left(
\begin{array}{cc}
d & a \\
b & c \\
\end{array}
\right)\left(
\begin{array}{cc}
d' & a' \\
b' & c' \\
\end{array}
\right)).\\
=& \vf(\left(
\begin{array}{cc}
dd'+ab' & da'+ac' \\
bd'+cb' & ba'+cc'\\
\end{array}
\right))\\
=&
\left(
\begin{array}{cc}
da'+ac' & bd'+cb'\\
ba'+cc'  & dd'+cb'\\
\end{array}
\right).	
\end{align*}
\end{example}
The map $\vf: M_n(R),.\longrightarrow M_n(R),*$ gives an isomorphism of rings ( see Lemma \ref{Lemme general construction of new product}).   
This new product structure restricted to the subring of upper triangular matrices over a finite field $F$ with the Frobenius map was already used in Habibi et al \cite{LCD1}.   The next proposition is a particular case of Lemma \ref{Lemme general construction of new product}.


\begin{proposition}
	Let $\varphi:M_n(R),.\longrightarrow M_n(R),*$ be the map
	define by $\varphi(a_{ij})=\theta^{i-1}(a_{ij})$.  Then $\varphi$ is an isomorphism of rings.
\end{proposition}
\begin{corollary}
	The inverse of a matrix $A\in M_n(R),*$ is the matrix $\varphi(B)\in M_n(R),*$ where $B$ is the usual inverse of $\varphi^{-1}(A)$. 
\end{corollary}
\begin{proof}
	Suppose that $\varphi^{-1}(A)B=I_n$, then $A*\varphi(B)=\varphi(\varphi^{-1}(A)B)=\varphi(I_n)=I_n$, as desired.
\end{proof}
In order to introduce a scalar product related to the new product of matrices, we are forced to generalize the above construction as follows.

Let $R$ be a ring, and 
for any $n,k\in \nn$, suppose we have an additive isomorphism 
$$
\varphi_{n,k}: M_{n,k}(R),+ \longrightarrow M_{n,k}(R),+
.$$
Now, if $A\in M_{n,k}(R)$ and $B\in M_{k,l}(R)$ we define
\begin{itemize}
	\item[(a)] $A\circ B=\varphi_{n,l}^{-1}(\varphi_{n,k}(A)\varphi_{k,l}(B))$.
	\item[(b)] $A*B=\varphi_{n,l}(\varphi_{n,k}^{-1}(A)
		\varphi_{k,l}^{-1}(B))$.
\end{itemize}
There are many ways of defining the additive maps $\varphi_{n,k}$ in general.  For instance, we could use any permutation of the $nk$ entries of the matrices in $M_{n,k}(R)$.
In this paper, we will consider an automorphism $\theta$ of the ring $R$ and construct an  additive  map $\varphi_{n,k}$ as follows:
$$
\varphi_{n,k}(A)_{ij}=(\theta^{i-1}(a_{i,j})) \quad 
{\rm where} \;\; A=(a_{i,j})\in M_{n,k}(R).
$$
With these notations we have:

\begin{lemma}
	\label{star and circ}
	Let $R$ be a ring, $A\in M_{n,k}(R), B\in M_{k,l}(R)$ and $\theta\in Aut(R)$.  With the above notations we have:
	\begin{enumerate}
	\item $(A\circ B)_{i,j}=\sum_r A_{ir}\theta^{r-i}(B_{r,j})$
	
	\item $(A * B)_{i,j}
	=\sum_r A_{ir}\theta^{i-r}(B_{r,j}).$
	
	\item If $n=k$ the map $\varphi :M_n(R),+,.\longrightarrow M_n(R),+,*$ defined by $\varphi((A_{ij}))=((\theta^{i-1})(A_{ij}))$ is a ring homomorphism.
\end{enumerate}
	\noindent  In (1) and (2) above, these products are additive on both variables but are linear only on the left.
	
\end{lemma}
\begin{proof}
	We have 
	\begin{align*}
		(A\circ B)_{ij}&=\varphi_{n,l}^{-1}((\varphi_{n,k}(A)\varphi_{k,l}(B)))_{ij}\\
		&=\theta^{1-i}(\sum_s\varphi_{n,k}(A)_{is}\varphi_{kl}(B)_{sj})\\
		&=\theta^{1-i}(\sum_s\theta^{i-1}(A_{is})\theta^{s-1}(B_{sj}))\\
		&=\sum_sA_{is}\theta^{s-i}(B_{sj}).
	\end{align*}
	A similar computation gives the second formula.
	
	(3) is left to the reader.
\end{proof}

Let us remark that in the above definitions of the maps $\varphi_{n,k}$ the indices $n,k$ just fix the size of the matrices we are working with.  In the sequel we will just write $\varphi$ and drop the indices.

We denote by $R^n$ the space of rows $M_{1,n}(R)$ and by $^nR$ the space of columns $M_{n,1}(R)$.   For $x\in R^n$, we denote $x^t\in\,^nR$ the transpose of $x$.
If we consider $\varphi_{1,n}$ and $\varphi_{n,1}$, $\circ$ and $*$ give two scalar products i.e. biadditive maps from $R^n\times\,^nR$ into $R$.   More explicitly we have, for $a,b\in R^n$ we have:
$$
a\circ b^t=\varphi_{1,1}^{-1}(\varphi_{1,n}(a)\varphi_{n,1}(b^t)) \quad {\rm and} \quad a*b^t=\varphi_{1,1}(\varphi_{1,n}^{-1}(a)
\varphi_{n,1}^{-1}(b^t)).
$$
In general, these scalar products are neither $R$-linear nor symmetric in the sense that, in general, for $a,b\in R^n$
$a\circ b^t\ne b \circ a^t$ and $a * b^t\ne b * a^t$.  Note that the orthogonality with respect to these scalar products is not the same as the classical orthogonality. 
Similarly, the orthogonalities for both the operations $\circ$ and $*$ are different.  Let $C\subseteq R^n$,
we define 
$$
^{\perp *}C=\{\ x\in R^n\mid x*c^t=0,\; \forall c\in C \} \quad {\rm and} \quad C^{*\perp} =\{ x\in R^n \mid c*x^t=0,\; \forall c\in C\}\\
$$
$$
^{\perp \circ}C=\{\ x\in R^n\mid x\circ c^t=0,\; \forall c\in C \} \quad {\rm and} \quad C^{\circ \perp} =\{ x\in R^n \mid c\circ x^t=0,\; \forall c\in C\}
$$
\begin{remarks}
We usually take $\varphi_{1,1}=Id$. In these cases 
\begin{enumerate}
	\item if $\varphi_{1,n}$ is a left linear  map from $R^n_R \longrightarrow R^n_R$. Then scalar products given by $*$ and $\circ$ are left linear  but just additive on the right.
	\item  if $\varphi_{1,n}$ is a right linear  map from $R^n_R \longrightarrow R^n_R$. Then scalar products given by $*$ and $\circ$ are right linear  but only additive on the left. 
	\item If $\varphi_{1,n}$ and $\varphi_{n,1}$ are the identity maps (on $R^n$ and $^nR$ respectively ) then $*$ is the usual scalar product.
	\item If $\varphi_{1,n}$ is the identity and for any $(b_1,\ldots,b_n)\in R^n,$ $\varphi_{1,n}((b_1,\ldots,b_n)^t)=(b_1^{*},\ldots,b_n^{*})^t$ where  $*$ is an involution on $R$, then $*$ is an Hermitian product. 
\end{enumerate}	
\end{remarks}
\begin{proposition} \label{pro:L-LCD} Let $C\subseteq R^n$, we have:
	\begin{enumerate}
		\item $C^{*,\perp}=(\varphi_{n,1}(\varphi_{1,n}^{-1}(C)^\perp))^t$.
		\item $\varphi_{1,n}^{-1}(^{\perp *}C)=\varphi_{n,1}^{-1}(C^t)^\perp)$.
		\item $^{\perp \circ}C=\varphi_{1,n}^{-1}(\varphi_{n,1}(C^t)^\perp)$.
		\item $C^{\circ \perp}=\varphi_{n1}^{-1}(\varphi_{1,n}(C)^\perp).$
	\end{enumerate}

\end{proposition}

\begin{proof}
(1)	Let the maps
	$$\varphi_{1,n} :M_{1,n} (R) \rightarrow M_{1,n} (R)\;;\;\varphi_{n,1} :M_{n,1} (R) \rightarrow M_{n,1} (R)\;{\rm and}\; \varphi_{1,1}=Id.$$
	\begin{align*}
		C^{*\perp}&=\{x\in R^n|\varphi_{1,1}(\varphi_{1,n}^{-1}(c)\varphi_{n,1}^{-1}(x^t))=0, \forall c\in C\}\\
		&=\{x\in R^n|\varphi_{1,1}(y\varphi_{n,1}^{-1}(x^t))=0, \forall y\in \varphi_{1,n}^{-1}(C)\}\\
		&=\{x\in R^n|\varphi_{1,n}^{-1}(C)\varphi_{n,1}^{-1}(x^t)=0\}\\
		&=(\varphi_{n,1}(\varphi_{1,n}^{-1}(C)^\perp))^t.
	\end{align*}
	Hence the result.

(2) This is left to the reader.

(3)We compute: \begin{align*}^{\perp \circ }C&=\{\ x\in R^n\mid x\circ c^t=0,\; \forall c\in C \}\\&=\{x\in R^n|\varphi_{1,1}^{-1}(\varphi_{1,n}(x)\varphi_{n,1}(c^t))=0 \}\\&=\{ \varphi_{1,n} ^{-1}(x)|x\varphi_{n,1}(c^t)=0 \}\\&=\varphi_{n,1}^{-1}(\varphi_{n,1} (C^t)^{\perp}).\end{align*}

(4) This is left to the reader.
\end{proof}

Similar results hold for the other scalar products.  Thanks to this proposition we can characterize  self-duality with respect to * via usual orthogonality.

\begin{corollary}
	A code $C$ is a self-dual code for $*$ (i.e. $C=C^{*\perp}$) if and only $\varphi_{n,1}^{-1}(C^t)=
	\varphi_{1,n}^{-1}(C)^{\perp}.$  Similarly if 
	$C=^{\perp *}C$, then we have $\varphi_{1,n}(C)=\varphi_{n,1}(C^t)^\perp$.
\end{corollary}

Similar results hold for the other scalar products.

\begin{examples}
	{\rm
		The map $\varphi_{1,3}: M_{1,3}(R) \longrightarrow M_{1,3}(R),$
		$\varphi_{1,3}(a_1,a_2,a_3)=(\lambda a_3,a_1,a_2)$, where $\lambda\in R$ and  $\varphi_{3,1}(b_1,b_2,b_3)^t=(b_2,b_1,b_3)^t$.
		The scalar product for $*$ given by
		$(a_1,a_2,a_3)*(b_1,b_2,b_3)^t= (\lambda a_3,a_1,a_2). (b_2,b_1,b_3)^t= \lambda a_3b_2+a_1b_2+a_2b_3$ \\
		Let $R=\F_4=\{0,1,\alpha, \alpha^{2}\}$ with $\alpha^2+\alpha +1=0$, $\lambda =\alpha$  and $C=\langle (1,0,1),(\alpha,1,0)\rangle.$ Then $(C)^{*,\perp}=\{(x,\alpha^{2}x,\alpha x)|x\in\F_4\}.$
	}

%
\end{examples}

%
%
%
%
%

Let $\theta$ be an automorphism of $R$.  For $n\ge 1$, $a=(a_1,\dots,a_n)\in R^n$, and $b=(b_1,\dots,b_n)\in R^n$, we define  $\varphi_{1,n}(a)=a$ and $\varphi_{n,1}(b^t)=
(b_1,\theta(b_2),\dots, \theta^{n-1}(b_n))^t$.
Explicitly we have $a*b^t=\sum_r a_r \theta^{r-1}(b_r)$ and  $a\circ b^t=\sum_r a_r \theta^{1-r}(b_r)$.

\begin{proposition}
	\label{properties of the scalar product}
	Let $C$ be a subset of $R^n$.  We have:
	\begin{enumerate}
		\item $^{\perp *}C$ is always an $R$-submodule of $R^n$ and $C^{* \perp}$ is always additive subgroup of $R^n$.
		\item Let $x,y \in R^n$, we have $x*y^t
		=x\cdot\varphi(y)$, where $\cdot$ stands for the usual dot product.
		\item We have $C\subseteq ^{\perp *}(C^{* \perp})$
		and $C\subseteq (^{\perp *}C^t)^{* \perp}$.
		\item[($3^{\prime}$)] We have $C\subseteq ^{\perp \circ}(C^{\circ \perp})$
		and $C\subseteq (^{\perp \circ}C^t)^{\circ \perp}$.
		\item We have $^{\perp *}C=\varphi(C^t)^{\perp}$ and $C^\perp =\varphi(C^{* \perp})$.
		\item[($4^{\prime}$)] We have $^{\perp \circ}C=\varphi(C^t)^{\perp}$ and $C^\perp =\varphi(C^{\circ \perp})$.
	\end{enumerate}  
\end{proposition}
\begin{proof}
	(1). These properties are direct consequences of the definitions.
	
	(2). Let $x,y\in R^n$, we have 
	$x*y^t=\sum_{i=1}^n 
	x_i\theta^{i-1}(y_i)=(x_1,\dots, 
	x_n)\varphi((y_1,\dots,y_n)^t)$.
	
	(3). Let  $(x_1,\dots,x_n) \in C^{* \perp}$
	then for any $c=(c_1,\dots,c_n) \in C$, we have $c* (x_1,\dots,x_n)^t=0$, i.e. $\sum c_i\theta^{i-1}(x_i)=0 $, i.e. $(c_1,\dots,c_n)\in ^{\perp,*}\{(x_1,\dots,x_n)\}$, this gives the first inclusion.  The second is obtained similarly.
	
	(3'),(4) and (4') are obtained similarly.
\end{proof}
 Let us remark that all the assertions of the above Proposition \ref{properties of the scalar product}, except (2), are valid in the general setting of $*$ and $\circ$. \\
\section{Left LCD codes and Applications}
In this section, we 
fix $\theta\in Aut(R)$.  We will use
 the $*$ multiplication for matrices as defined 
in Lemma \ref{star and circ}.  Analogues of the results that we will obtain are also true for the $\circ$ multiplication.   We  
give the definition of a $*$-LCD code and a necessary 
and sufficient condition for a code is a $*$-LCD code.
\begin{definition}
\label{defi:LCD}
	A left linear code $C$ is  $*$-LCD  if $C\cap C^{*\bot}=\{0\}$ $(C\cap ^{\perp*}C=\{0\})$.
\end{definition}

If $C$ is a left linear code, we can define it via a set of row bases $\{c_1,\dots,c_k\}\subseteq R^n$.   So if $H$ is the $k\times n$ matrix defined by these vectors, we have that $C=R^kH$.  Since the map $\varphi$ is a bijection, we can also define the code $C$ by $C=R^k*G$ for a matrix  $G\in M_{k,n}(R)$. 

\begin{lemma}
	Let $R$ be a commutative ring, $y\in R^k$ and $G\in M_{k,n}(R)$, we have 
	$(y*G)^t=\varphi(G)^t*\varphi^{-1}(y^t).$
\end{lemma}
\begin{proof}
We compute 
$(y*G)^t=(y\varphi(G))^t=\varphi(G)^ty^t
=\varphi(G)^t\varphi(\varphi^{-1}(y^t))\\
=\varphi(G)^t*\varphi^{-1}(y^t)$.	
 This concludes the proof.
\end{proof}
\begin{theorem} \label{thm:LCD}
Let $R$ be a commutative field and $G\in M_{k,n}(R)$.  The code $C=R^k*G$ is a left *-LCD code if and only if 
	$G*\varphi(G)^t\in M_n(R)$ is invertible.
\end{theorem}
\begin{proof}
	Assume that the matrix $G*\varphi(G)^t\in M_n(R)$ is invertible,  and let $mathbf{c}\in C\cap ^{\perp,*}C$.  There exists $u\in R^k$ such that $mathbf{c}=u*G$.   Since $mathbf{c}\in ^{\perp,*}C$, we have that for every $y\in R^k$, $mathbf{c}*(y*G)^t=0$.  We thus get that $0=(u*G)*(y*G)^t$.  The above lemma then gives
	$0=u*G*\varphi(G)^t*\varphi^{-1}(y^t) $.   Since $y\in R^k$, our assumption gives that
	$u=0$ and hence $c=0$.
	Conversely, let us show that if $C\cap ^{\perp,*}C=\{0\}$ then $G*\varphi(G)^t$ is invertible.  Assume to the contrary that $G*\varphi(G)^t$ is not invertible and let $v\in R^k$ be such that $v*G*\varphi(G)^t=0$.  
	For any $e=e'*G\in C$, we have $v*G*e^t=v*G*(e'*G)^t*=v*G*\varphi(G)^t*\varphi^{-1}
	(e'^t)=0$.  We thus have that $v*G\in C\cap ^{\perp,*}C$, a contradiction.  	
\end{proof}

\begin{remarks}
{\rm
	The above result has analogues in all the three remaining orthogonals viz. $C^{*,\perp}, ^{\perp, \circ}C$ and $C^{\circ, \perp}$.   Notice that for the right orthogonals $C^{*,\perp}, C^{\circ,\perp}$ we have to use a right linear code $C$ defined by a matrix $H\in M_{n,k}(R)$, i.e. $C=HR^k$.
}
\end{remarks}

\subsection {Results and computation}
	In this subsection, we present examples of good left LCD  codes and dual code $C^{\perp*}$ form the from   Theorem \ref{thm:LCD} and  Proposition \ref{pro:L-LCD}  ( $\varphi^{-1}(C^{\perp})=C^{*\perp}$) over $(GF(4))$,  $(GF(8))$, $(GF(9))$ and  $(GF(16))$ . These codes  ( left LCD) are either optimal or have the same parameters as best known linear codes available in the database \cite{LCD3}.\\
\begin{table}[h]
	\caption{Examples of best known and optimal left-LCD codes over $GF(4)$ and $GF(8)$ }
	\begin{tabular}{|c|c|c|} 
		\hline
		\textbf{Generator Matrix} & $[n,k,d]_q$ & \textbf{Type of Codes}\\
		\hline
			$ \left(\begin{array}{ccccccc} 1&0&0&0&w^5&w^2&1\\0&1&0&0&w^2&w^2&1\\0&0&1&0&w^4&w^6&w^6\\
		0&0&0&1&w^6&1&w^6 \end{array}\right)$ & $[7,4,4]_8$ & $C^{\perp,*}$ Left-LCD \\
		\hline
		$ \left(\begin{array}{ccccccc} 1&0&0&0&w^2&0&w^2\\0&1&0&0&1&1&0\\0&0&1&0&w^2&w&w^2\\0&0&0&1&0&w^2&1 \end{array}\right)$ & $[7,4,3]_4$ & $C^{\perp,*}$ Left-LCD \\
		\hline
	$ \left(\begin{array}{ccccccc}
	1&0&0&w^2&w^2&w^2&1\\0&1&0&1&1&0&1\\0&0&0&1&w^2&w^2&w^2 \end{array}\right)$ & $[7,3,4]_4$ & $^{\perp *}C$\\
		\hline
		
		$ \left(\begin{array}{cccccccccc}
		1&0&0&w^2&w^2&w^3&w^5&w^6&w^4&w\\0&1&0&w^5&w^4&w^2&w&w^4&w^3&w^6\\
		0&0&1&w&w^5&w^6&w^6&w^4&w&w^4
		\end{array}\right)$ & $[10,3,8]_8$ & $C^{\perp *}$ Left-LCD\\
		\hline
			$ \left(\begin{array}{ccccccc}
		1&0&0&w^3&w^3&w^2&w^6\\0&1&0&w^4&w^5&1&w^2\\0&0&1&w^4&w^4&w^4&w^4 \end{array}\right)$ & $[7,3,4]_8$ & $^{\perp *}C$\\
		\hline
		$ \left(\begin{array}{ccccccccccc}
	1&0&0&0&0&0&w^2&1&1&w^2&w\\0&1&0&0&0&0&w^2&w&0&w&0\\0&0&1&0&0&0&0&1&w^2&w^2&1\\0&0&0&1&0&0&0&0&w&w^2&1\\0&0&0&0&1&0&w^2&0&0&1&w^2
	\\0&0&0&0&0&1&w^2&w^2&w&1&w^2 \end{array}\right)$ & $[11,6,4]_4$ & $C^{\perp,*}$ Left-LCD\\
		\hline	
	$ \left(\begin{array}{ccccccccccc}
	1&0&0&0&0&w&1&w^2&0&0&0\\0&1&0&0&0&w^2&0&1&w^2&w^2&0\\0&0&1&0&0&0&0&w&1&w&1\\0&0&0&1&0&0&0&w&0&w^2&w^2\\0&0&0&0&0&0&1&1&0&w&w^2
 \end{array}\right)$ & $[11,6,4]_4$ & $^{\perp *}C$\\
	\hline	
	\end{tabular}
\end{table}
\begin{center}
\begin{table}
	\caption{Examples of best known and optimal Left  LCD codes over  $GF(9)$ and $GF(16)$ }
	\begin{tabular}{|c|c|} 
		\hline
		\textbf{Generator Matrix} & $[n,k,d]_q$\\
		\hline
		$ \left(\begin{array}{ccccccccccccccccccc} 1&0&0&0&w^7&0&w&w^5&w&2&w^3&w^3&w^3&w^3&0&1&w^2&w^3&1\\
		0&1&0&0&w^6&0&w^2&w^3&w^2&0&w^6&w^2&w&1&w^5&w^3&0&2&1\\
		0&0&1&0&1&0&2&w^3&0&w^5&w^3&0&w^3&1&2&w^6&w^6&w^5&w^5\\
		0&0&0&1&w^5&0&w^3&w^2&w&w^3&1&w^2&w^7&w^2&w^3&2&w^5&w^3&w^7\\
		0&0&0&0&0&1&w^7&2&w^2&w^3&1&w^7&w^5&w^5&w&w^6&w^7&w&w 
		\end{array}\right)$
		 & $[19,5,12]_{9}$ \\
		\hline
		
		$ \left(\begin{array}{ccccccccccccccccccc} 1&0&0&0&w^7&0&w&w^5&w&2&w^3&w^3&w^3&w^3&0&1&w^2&w^3&1\\
		0&1&0&0&w^2&0&w^6&w&w^6&0&w^2&w^6&w^3&1&w^7&w&0&2&1\\
		0&0&1&0&0&1&0&2&w^3&0&w^5&w^3&0&w^3&1&2&w^6&w^5&w^5\\
		
		\end{array}\right)$
		& $[19,3,13]_{9}$ \\
		
		\hline
			$ \left(\begin{array}{ccccccccccccccccccc}
			1&0&0&0&0&1&w^7&2&w^2&w^3&1&w^7&w^5&w^5&w&w^6&w^7&w&w \\ 0&1&0&0&w^5&0&w^3&w^7&w^3&2&w&w&w&w&0&1&w^6&w&1\\
		0&0&1&0&w^2&0&w^2&w^3&w^2&0&w^6&w^2&w&1&w^5&w^3&0&2&1\\
		0&0&0&1&1&0&2&w&0&w^7&w^w&0&w^2&1&2&w^2&w^2&w^7&w^7\\
		0&0&0&0&w^5&0&w^3&w^2&w&w^3&1&w^2&w^7&w^2&w^3&2&w^5&w^3&w^7\\
		
		\end{array}\right)$
		& $[19,5]_{9}$ \\
		\hline
		
		$ \left(\begin{array}{ccccccccccccccccc}
		1&0&0&0&0&0&0&w^10&w^4&w^6&w^4&w^{13}&w&w^9&w^{13}&w^5&1\\
		0&1&0&0&0&0&w^{12}&w^{10}&w^9&w^{11}&w^{11}&w^{12}&w^2&w^{13}&w^{12}&w^{12}&w^{12}\\
		0&0&1&0&0&0&w^{14}&w^7&w^8&w^9&w^{10}&w&w^{10}&w^5&w^{12}&1&w\\
	   0&0&0&1&0&0&w^3&w^3&w^9&w^{12}&w^6&w^3&w^2&w^3&1&w^3&1\\
	   0&0&0&0&1&0&w^5&w^2&w^{12}&w^5&w^{10}&w^{12}&w^4&w^6&w^9&w^{10}&w^{12}\\
	   0&0&0&0&0&1&w^{10}&1&w^{10}&w^5&w^{10}&1&w^{10}&1&w^5&1&1
		\end{array}\right)$ & $[17,6,10]_{16}$\\
		\hline	
		
	\end{tabular}
\end{table}
\end{center}
\section {Acknowledgement}
This work was supported by  The International Mathematical Union (IMU Abel Visiting Scholar Program 2024). We would like to thank Professor  Steven Dougherty  for correcting misprints
in a previous version of this paper.

\bibliographystyle{abbrv}

\end{document}